\documentclass[a4paper,12pt]{amsart}
\input epsf
\usepackage{epsfig}
\usepackage{color}
\DeclareMathAlphabet\mathbold{OML}{cmm}{b}{it}
\numberwithin{equation}{section}
 
\newtheorem{theorem}{Theorem}[section]
\newtheorem{lemma}{Lemma}[section]

\def\b1{{\mathbf 1}}
\def\bv{{\mathbf v}}
\def\bu{{\mathbf u}}
\def\bw{{\mathbf w}}
\def\bb{{\mathbf b}}
\def\bg{{\mathbf g}}
\def\bh{{\mathbf h}}
\def\br{{\mathbf r}}

\def\bd{{\mathbf d}}

\def\bp{{\mathbf p}}

\def\bx{{\mathbf x}}

\def\bbf{{\mathbf f}}

\def\balpha{{\boldsymbol \alpha}}
 
 \newcommand{\avg}[1]{\left\{\hspace{-0.045in}\left\{#1\right\}
\hspace{-0.045in}\right\}}
\newcommand{\jump}[1]{\left[\hspace{-0.025in}\left[#1\right] 
\hspace{-0.025in}\right]}

\newcommand{\B}{{\mathcal B}}

\newcommand{\J}{{\mathcal J}}

\renewcommand{\O}{{\mathcal O}}

\title[A combined preconditioner for nonsymmetric systems]
{A combined preconditioning strategy for nonsymmetric systems}
\author{Blanca Ayuso de Dios} 
\address{Computer,~Electrical and Mathematical Sciences \& Engineering Division (CEMSE) 
King Abdullah University of Science and Technology
Thuwal 23955-6900, Kingdom of Saudi Arabia} 
 \email{blanca.ayusodios@kaust.edu.sa}
\author{Andrew T. Barker} 
\address{Center for Applied Scientific Computing,
             Lawrence Livermore National Laboratory,
             7000 East Avenue, Mail Stop L-560,
             Livermore, CA 94550, U.S.A.}
 \email{barker29@llnl.gov}

\author{Panayot S. Vassilevski}

\address{Center for Applied Scientific Computing,
             Lawrence Livermore National Laboratory,
             7000 East Avenue, Mail Stop L-560,
             Livermore, CA 94550, U.S.A.}
             \email{panayot@llnl.gov}


 
 \date{\today}
 \thanks{The work of the first author was partially supported by Spanish MEC project MTM2011-27739-C04-04. The work of the second and third authors was performed under the auspices of the U.S. Department of Energy by Lawrence Livermore National Laboratory under Contract DE-AC52-07NA27344.}

\begin{document}
 
\begin{abstract}
We present and analyze a class of nonsymmetric preconditioners within a normal (weighted least-squares) 
matrix form for use in GMRES to solve  
nonsymmetric matrix problems that typically arise in finite element discretizations.  
An example of the additive Schwarz method applied to nonsymmetric but definite matrices 
is presented for which the abstract assumptions are verified. 
A variable preconditioner, combining the original nonsymmetric one and a  
weighted least-squares version of it, is shown to be convergent and provides a 
viable strategy for using nonsymmetric preconditioners in practice. 
Numerical results are included to assess the theory and the performance of the proposed preconditioners. 
\end{abstract}
\maketitle
\section{Introduction}

The numerical approximation of most phenomena in science and technology requires the solution of linear or nonlinear algebraic systems. 
Preconditioning is one of the main techniques that combined with a proper iterative method allows for reducing substantially the cost of solving those systems. Many efforts are usually devoted to design proper preconditioning strategies that allow for efficient and fast solution of the resulting algebraic systems \cite{toselli-libro,vassi-book}. The development of preconditioners is very often guided by the properties of the underlying problem and it sometimes can even dictate particular aspects that should be accounted for, when devising the numerical discretization of the continuous problem (as for instance in \cite{fosls0,fosls1}).\\


Even for linear problems, the design and analysis of preconditioners for the linear systems is far from being complete. For symmetric and coercive problems, a reasonable discretization yields a linear system $A_0 \bx = \bb$ with $A_0$ symmetric and positive definite (s.p.d). In such cases, as it is well known, the spectral information of the matrix itself dictates completely the convergence of the method. Therefore a preconditioner $B_0$ would be uniform (and could be turned into an optimal preconditioner) if it captures completely such spectral information; in other words,  if $B_0$ is spectral equivalent to $A_0$. \\

However, for linear systems $A \bx = \bb$, with nonsymmetric coefficient matrix $A$, the design of effective preconditioners does not admit a general recipe, at least at the present time. Likewise, there is no general iterative solver, and furthermore, there is no general theory that could be used to explain the success of a particular preconditioner when it is indeed efficient. In most cases, the spectral information does not provide significant information that could guide the development of any good preconditioner. Field of values have shown some utility in certain circumstance, but have also shown many limitations~\cite{ernst, Starke, klawon0}.  At the moment it might seem that  each particular problem has to be studied separately and a problem-dependent, discretization-dependent  preconditioning strategy had to be devised. Even when such preconditioning can be designed, its understanding and analysis are tasks that in most cases are out of reach.\\

In this paper, we focus on a particular situation, where $A$ is nonsymmetric but still positive definite. 
The motivation and application comes from a nonsymmetric discontinuous Galerkin discretization of an elliptic problem \cite{IIPG04}. In \cite{paola-blanca1,paola-blanca2}, additive and multiplicative Schwarz preconditioners were developed for the solution of the resulting algebraic system. In both works, the authors show that the GMRES convergence theory cannot be applied for explaining the convergence since the preconditioned system does not satisfy the {\it sufficient} conditions required by such theory. However, such discretizations are used in practice and have already shown to have some advantages when approximating advection-diffusion problems~\cite{hsusch, am0} and more recently,  in the design of methods for some more complex nonlinear problems \cite{suli-ortner, suli2}. In \cite{az}, the authors introduce a solver methodology based on the idea of subspace correction for this type of discretization for elliptic problems (see also \cite{ahzz14} for the extension to problems with jump coefficients), providing the analysis of the resulting iterative methods without using any GMRES theory.  In this paper we want to examine, in a more general algebraic abstract framework (that in particular will apply to the type of methods discussed above), the issue of providing some convergence theory for a preconditioner based on the classical (but nonsymmetric) Schwarz preconditioner to be used within GMRES. The ultimate goal is to obtain some insight on how to improve and tune the preconditioner. \\

Here, in a first stage we consider two preconditioners for $A$;  a classical additive Schwarz preconditioner $B$, which is nonsymmetric, and a symmetric preconditioner $Z$ that basically uses actions of the additive Schwarz preconditioner $B$ and its adjoint. Both will be shown to have their pros and cons. For the former, the non-symmetry of $B$ and of $B^{-1}A$, precludes developing any theory from which to extract  either some a-priori information on the convergence or to provide some guidelines on how the preconditioner could be improved or even be designed. The latter, while allowing for developing a convergence theory, will be shown to be not the most efficient possible option, although the a-priori information on convergence could be of special value depending on the application. Other symmetrizing strategies  for classical Schwarz methods, different from the one introduced here, have been already considered in literature by other authors \cite{holst97}.\\ 

At first sight, the underlying message that one might obtain from the analysis of this first part, is that enforcing the symmetry of the preconditioner for a nonsymmetric matrix might result in the very end, in a {\it wasted effort}.  We believe this might be the case in many situations, and we also think it is relevant and important to point it out.  At the same time, we do believe that the results obtained for analyzing the preconditioner, are of independent interest (also because of its simplicity), and might  provide some basis  (as it had happened already here) or insights for further development of solvers for nonsymmetric systems. \\

In a second stage of the present paper, we introduce a variable preconditioner $\mathcal{B}$ that is constructed by considering a linear combination of two given (general) preconditioners $ B $ and $ Z $ so that in a sense it tries to integrate and exploit the best of each of them. 
We describe the construction of this variable preconditioner to be used in GMRES, explaining how the coefficients in its definition are determined at each iteration inside GMRES. We  show that from the construction of $\mathcal{B}$ we immediately can deduce (theoretically) a convergence estimate that guarantees better performance of the resulting solver at a fixed iteration. We demonstrate numerically  that the new preconditioner outperforms the symmetric preconditioner $Z$ and always converges faster. We also include the construction of a variant of the combined preconditioner, for which one can indeed guarantee faster convergence than for the GMRES using only the $Z$ preconditioner.  However, this variant is more expensive, and therefore, although convergence of the whole sequence of iterates can be shown, its advantages and its possible tuning with respect to the original combined preconditioner need further study and will be  subject of future research.\\

The theory is illustrated with extensive numerical experiments, in which we also study the performance of all the considered preconditioners. They  are all implemented in parallel to fully take advantage of having considered preconditioners based on additive Schwarz methods. 
In the numerical tests, we do observe that the combined preconditioner  requires less GMRES-iterations to achieve  convergence  than the the classical additive Schwarz preconditioner $B$. However, in this particular case, 
each iteration for the combined preconditioner is more costly, which in the end, makes $B$ perform slightly better in terms of execution time. From these observations, it might be inferred that the new combined preconditioner $\mathcal{B}$ might be more competitive in settings where each iteration is expensive, so that the savings in iteration count can make up for the high cost per iteration.\\

Although we have focused on the nonsymmetric but positive definite case, we believe the ideas presented in the paper might be useful and possibly extended to more complex problems, including the indefinite case. This issue will be subject of future research.\\

The proposed strategy of the use of combined preconditioner inside GMRES is certainly related to the {\it Flexible} GMRES (FGMRES) introduced in \cite{fgmres00}, and the more recent works on {\it Multi-preconditioned} GMRES (MGMRES) \cite{daniel0,daniel1}. 
The idea of using several preconditioners instead of only one with GMRES has been already introduced with different viewpoints (and still different from the one considered here).  Saad proposed to use a different preconditioner at each GMRES iteration in \cite{fgmres00}, giving rise to what is coined today as {\it Flexible} GMRES (FGMRES). For an earlier result that uses preconditioners that can change at every iteration step (referred to as a variable--step preconditioners), we refer to \cite{AV91}. More recently, it has been proposed the {\it Multi- Preconditioned} GMRES (MPGMRES), as a further evolution of the multi-preconditioned CG method introduced in \cite{greif}. The idea of both works is to enlarge the Krylov spaces over which GMRES minimizes the residual norm, by considering an optimal combination of different preconditioners, and therefore using several preconditioners simultaneously at each iteration.
The authors have considered in \cite{daniel0} complete (using all possible search directions generated by the different preconditioners) and truncated (where  some directions are discarded to reduce the cost)  versions of the MPGMRES. 
While the strategy proposed in this paper certainly has similarities with these previous works, it has also many differences, including the motivation. 
Unlike the present paper, the particular preconditioners used in \cite{daniel0, daniel1} do not play any particular role. Here, the main motivation comes from the observation that for some discretizations frequently used in applications, the simplest possible domain decomposition preconditioner cannot be guaranteed to be convergent. A symmetric version of it, the preconditioner $Z$ can be shown to converge but its performance seems to be far from being optimal. Therefore, the idea of trying to construct an optimal combination of both that could retain the good  performance of the nonsymmetric preconditioner and for which some convergence estimate could be given (even if it would be just an upper bound).  Here, the focus is on combining in optimal way very specific preconditioners rather than allowing any preconditioner. 
The possible comparison and further combination of both approaches (the combined preconditioned and the MPGMRES) are left as a subject of future research.

The outline of the paper is as follows. Section \ref{sec:2} contains a description of the problem and the original motivation of it.
  In Section \ref{sec:3}, we construct the preconditioner $Z$ and present the convergence analysis. The combined preconditioner and its variant are introduced and analyzed in Section \ref{sec:4}. Finally in Section \ref{sec:5}, we consider a particular application and we provide numerical experiments that verify the developed theory and assess the performance of the preconditioner.

 \section{Problem formulation and Basic Notation}\label{sec:2}
We are interested in preconditioning a given system of linear equations 
\begin{equation}\label{system:0}
A \bx = \bb, \qquad A\in \mathbb{R}^{n\times n} \quad \bx,\, \bb \in \mathbb{R}^{n}\;,
\end{equation}
with $A$ being non--symmetric but definite and $n$  is assumed to be large. For the applications we have in mind, $A$ comes from a finite element  discretization of some partial differential operator and therefore is sparse and structured. 
With a small abuse of notation, throughout the paper we will use (the same notation) $A$ to denote both an operator and its matrix representation, since it will be clear from the context in all circumstances.
 
 We also assume that $A$ is ill-conditioned and that therefore a good preconditioner $B$ is required to solve efficiently  system \eqref{system:0} by an iterative method.
A simple option 
is to construct such $B$ as the classical additive Schwarz preconditioner coming from $A$. More precisely,  we denote by $I_k$, $k=1,\;\dots,\; N_s $, a set of rectangular matrices, such that $I_k$  extends a local vector $\bv_k$ to a global vector 
$I_k \bv_k$ with zero entries outside its index set. Also, let $I_c$ be an interpolation matrix that maps a coarse vector $\bv_0 = \bv_c$ 
to a global vector $I_c \bv_c$. Then, the additive Schwarz preconditioner exploits the local matrices $A_k = I^T_k A I_k$,
principal submatrices of $A$, and the coarse matrix $ A_c$ defined as $A_c=I^T_c A I_c$. The inverse of the additive 
Schwarz preconditioner $B$ takes the following familiar form:
\begin{equation}\label{prec:B}
B^{-1} = I_c A^{-1}_c I_c^T + \sum\limits^{N_s}_{k=1} I_k A^{-1}_k I^T_k\;. 
\end{equation}
Obviously, since $A$ is nonsymmetric the resulting additive 
Schwarz preconditioner $B$ is also nonsymmetric. Therefore, for the solution of the resulting preconditioned system $AB^{-1} \bu = \bb, B^{-1} \bu = \bx$, one has to use any of the iterative methods for nonsymmetric systems, such as the {\it Generalized Minimal Residual} (GMRES). For analyzing the convergence of the resulting iterative method (for the preconditioned system) one has to resort to one of the available and non-optimal GMRES theories. In the Domain Decomposition framework, the  GMRES convergence theory of Eisenstat et. al. \cite{elman} is generally used. 
In particular, to derive (a-priori) any conclusion on the performance of the preconditioner $B$, this theory requires some control on the coercivity of $AB^{-1}$  (in some inner product). Therefore, at least in theory,  using $B$ directly as a preconditioner for $A$ might not be successful.\\

Still, we would like to utilize the actions of $B^{-1}$ to define a preconditioner, say $Z$,  for $A$, for which some  bounds on the rate of convergence can be a-priori determined. 
In the next section, we show how such a preconditioner $Z$ can be constructed (and analyzed) by exploiting the fact that although $A$ is nonsymmetric, it is positive definite in some inner product. We also compare numerically, in Section \ref{sec:5}, the performance of the constructed  preconditioner $Z$ with the original nonsymmetric additive Schwarz preconditioner $B$. As we will show, even if a theory can be developed for $Z$ it might not be the most efficient option.\\

We now state our basic assumption regarding the matrix $A$. More specifically, we assume that there is an s.p.d. matrix $A_0$ such that $A$ and $A_0$ are related by the following basic assumption:\\

\vskip 0.2cm
{\bf Assumption (H0):} {\it Let $A\, \in\, \mathbb{R}^{n\times n}$ be nonsymmetric but definite and let $A_0\, \in \mathbb{R}^{n\times n} $ be s.p.d.
We say that the pair of matrices $(A\, ,\, A_0)$  satisfy 
{\bf Assumption (H0)} with constants $(c_0,c_1)$ if they do satisfy the following coercivity and boundedness estimates:
\begin{equation}\label{A:coercivity}
\bv^TA \bv \ge c_0 \; \bv^T A_0 \bv\quad \text{for all }\bv \in \mathbb{R}^{n}\;,
\end{equation}
\begin{equation}\label{A:boundedness}
\bw^TA\bv \le c_1 \; \sqrt{\bv^TA_0 \bv}\sqrt{\bw^TA_0 \bw} \quad \text{ for all }\bv,\bw \, \in\, \mathbb{R}^{n}\;.
\end{equation}
}
\vskip 0.2cm

Apart from the above assumptions, the matrix $A$ we consider is assumed to be highly non-normal, which precludes the use of iterative algorithms and GMRES convergence theories that exploit the eigenvalue information.

\section{An abstract result}\label{sec:3}
In this section we present the construction and give the analysis of a preconditioner for $A$ that basically only uses the actions of the additive Schwarz method. We start by proving two Lemmas that will be required for our subsequent analysis and derivation.\\

The next Lemma shows that for any pair of matrices $(A,A_0)$ satisfying {\bf (H0)} with constants $(c_0,c_1)$, the corresponding pair $(A^{-1},A_0^{-1})$ (consisting of their inverses)  also satisfies  {\bf (H0)} with constants $(c_3,c_4)$ that depend only on $c_0$ and $c_1$.

\begin{lemma}\label{le:0}
Let $A\, \in\, \mathbb{R}^{n\times n}$ be nonsymmetric but definite and let $A_0\, \in \mathbb{R}^{n\times n} $ be s.p.d
Let $(A, A_0)$  be a pair of matrices that satisfy assumption {\bf (H0)} with constants $(c_0,c_1)$ (in particular, $A\, \in\, \mathbb{R}^{n\times n}$ is nonsymmetric but definite and $A_0\, \in \mathbb{R}^{n\times n} $ is s.p.d). Then,
the pair  $(A^{-1}, A^{-1}_0)$ also satisfies assumption {\bf (H0)} with constants $\left(\frac{c_0}{c_1^{2}},c^{-1}_0\right)$; that is,
\begin{equation}\label{le0:1}
\bv^T A^{-1} \bv \ge \frac{c_0}{c_1^2}\; \bv^T A^{-1}_0 \bv, \quad \text{for all }\bv \in \mathbb{R}^{n},
\end{equation}
\begin{equation}\label{le0:2}
\bw^T A^{-1} \bv \le \frac{1}{c_0}\; \sqrt{\bv^T A^{-1}_0 \bv} \sqrt{\bw^T A^{-1}_0 \bw}, \quad \text{ for all }\bv,\bw \in \mathbb{R}^{n}.
\end{equation}
\end{lemma}
\begin{proof}
We first show the boundedness estimate \eqref{le0:2}. We define the matrix $Y: = A^{-\frac{1}{2}}_0 A A^{-\frac{1}{2}}_0$. Then,  \eqref{A:coercivity} and \eqref{A:boundedness}  imply (or read) that $Y$ satisfies:
\begin{align}
\bv^T Y \bv &\ge c_0 \; \|\bv\|^2,  \qquad  \text{for all}\, \bv \in \mathbb{R}^{n}\;, && \label{le0:3}\\
\bw^T Y \bv &\le \; c_1 \|\bv\| \|\bw\|, \qquad  \text{for all}\, \bv, \bw\, \in \mathbb{R}^{n}\;. \label{le0:4}&&
\end{align}
The positivity  \eqref{le0:3} of $Y$  guarantees the existence of $Y^{-1}$ and so taking  $\bv:= Y^{-1} \bw$ in \eqref{le0:3}, and using the symmetry of the standard $\ell^{2}$-inner product of two vectors together with the Cauchy-Schwarz inequality, we find
\begin{equation*}
c_0\;\|Y^{-1} \bw\|^2 \le \bw^T Y^{-T} \bw = \bw^T Y^{-1} \bw \le \|\bw\| \|Y^{-1} \bw\|,
\end{equation*}
which shows that $\|Y^{-1} \bw \| \le \frac{1}{c_0}\; \|\bw\|$, that is, the boundedness of $Y^{-1}$ in the $\ell^{2}$-norm:
\begin{equation}
\|Y^{-1}\| \le \frac{1}{c_0}.
\end{equation}
 In other words we have shown that
\begin{equation*}
\bw^T A^{\frac{1}{2}}_0 A^{-1} A^{\frac{1}{2}}_0 \bv = \bw^T Y^{-1} \bv \le \frac{1}{c_0}\; \|\bv\| \|\bw\| \quad \forall\, \bv,\bw \in \mathbb{R}^{n}\;.
\end{equation*}
Setting now in the above equation $\bv := A^{-\frac{1}{2}}_0 \bv$ and $\bw := A^{-\frac{1}{2}}_0 \bw$,  we reach the desired boundedness estimate \eqref{le0:2} for $A^{-1}$ in terms of $A^{-1}_0$.\\

The positivity estimate \eqref{le0:1} can be shown as follows. 
On the one hand, the boundedness  \eqref{le0:4} of $Y$  with $\bv=\bv$ and $\bw = Y^{-1} \bv$ gives
\begin{equation*}
\|\bv\|^2 = \bv^T Y (Y^{-1} \bv) \le 
c_1 \; \|Y^{-1} \bv\| \|\bv\| \qquad \forall\, \bv \in \mathbb{R}^{n}\;.
\end{equation*}
which readily implies
\begin{equation}\label{le0:5}
 \|Y^{-1} \bv\| \geq \frac{1}{c_1} \|\bv\| \qquad \forall\, \bv \in \mathbb{R}^{n}\;.
 \end{equation}
  On the other hand, using the positivity estimate \eqref{le0:3} of $Y$, we have
\begin{equation*}
\bv^T Y^{-1} \bv = (Y^{-1} \bv)^T Y^T (Y^{-1} \bv) =  (Y^{-1} \bv)^T Y (Y^{-1} \bv) \ge c_0 \; \|Y^{-1} \bv\|^2 \quad \forall\, \bv \in \mathbb{R}^{n}\;.
\end{equation*}
Then, the above relation together with estimate \eqref{le0:5} give the following  positivity estimate for $Y^{-1}$ :
\begin{equation*}
\bv^T Y^{-1} \bv \ge \frac{c_0}{c_1^{2}}\; \|\bv\|^2 \quad \forall\, \bv \in \mathbb{R}^{n}\;.
\end{equation*}
Now, setting in the above estimate
 $\bv := A^{-\frac{1}{2}}_0 \bv$, we obtain the coercivity relation \eqref{le0:1} and conclude the proof.
\end{proof}

\vskip 0.2cm

Next, let $B_0$ denote the  s.p.d. additive Schwarz preconditioner of $A_0$, whose inverse is defined as:
\begin{equation}\label{prec:B0}
B_0^{-1} = I_c A^{(0)^{-1}}_c I_c^T + \sum\limits^{N_s}_{k=1} I_k A^{(0)^{-1}}_k I^T_k\;. 
\end{equation}

Note that since $(A,A_0)$ satisfy assumption {\bf (H0)} with constants $(c_0,c_1)$, this immediately implies that for each $k=1,\ldots N_s$ the family of pairs $(A_k, A^{(0)}_k)$ with matrices defined by 
\begin{equation*}
A_k := I^T_k A I_k \quad \mbox{and } \quad A^{(0)}_k := I^T_kA_0 I_k \qquad k=1,\ldots N_s\;,
\end{equation*}
 also satisfy assumption {\bf (H0)} with the same constants. The same is also true for the {\it coarse} pair of matrices $(A_c, A^{(0)}_c)$, where $A_c=I_c^T A I_c$ and $A^{(0)}_c=I_c^T A_0 I_c$.
Then, applying Lemma \ref{le:0} to each of these pairs, we have that the corresponding pair of their 
respective 
inverses and hence the pair with the product matrices $\left(I_k A^{-1}_k I^T_k\, , \, I_k A^{(0)^{-1}}_k I^T_k\right)$ satisfies {\bf (H0)} with constants $(c_0c_1^{-2}, c^{-1}_0)$ (i.e, \eqref{le0:1} and \eqref{le0:2}).  The latter implies that the inverses of the 
additive Schwarz preconditioners $B^{-1}$ (as defined in \eqref{prec:B}) and  $B^{-1}_0$ (as defined in \eqref{prec:B0}), are related in the same way (as their individual terms
$I_k A^{-1}_k I^T_k$ and $I_k A^{(0)^{-1}}_k I^T_k$). That is: $(B^{-1},B_0^{-1})$ also satisfy {\bf (H0)} with constants $(c_0c_1^{-2}, c^{-1}_0)$.
Applying once more Lemma \ref{le:0}, we straightaway deduce that the pair $(B,B_0)$ also satisfy {\bf (H0)}, now with constants $(c_0^{3}c_1^{-2},c_1^{2}c_0^{-1})$.\\ 

Now, since $B_0$ is the classical s.p.d. additive Schwarz preconditioner for the s.p.d $A_0$, $B_0$ and $A_0$ can be shown to be spectrally equivalent: there exists $\gamma_0,\gamma_1>0$ such that
\begin{equation}\label{b0:a0}
 \gamma_0 \bv^{T}B_0 \bv \leq  \bv^{T}A_0 \bv \leq \gamma_1 \bv^{T} B_0 \bv \quad \forall\, \, \bv \in \mathbb{R}^{n}\;,
 \end{equation}
where the constants $\gamma_0$ and $\gamma_1$  might depend on the parameters of the discretization and the problem. $B_0$ would be optimal if neither $\gamma_0$ nor $\gamma_1$ depend on the discretization parameters (or size of the system $n$).

Using this extra information, it is straightforward to deduce that the pair $(B,A_0)$ also satisfy {\bf (H0)} with constants $(\beta_0,\beta_1)$ that depend only on $c_0,c_1,\gamma_0$ and $\gamma_1$. 
All these observations are summarized in the following Lemma:\\
\begin{lemma}\label{le:1}
Let $(A,A_0)$ satisfy assumption {\bf (H0)} with constants $( c_0,c_1)$.
Let $B$ be the additive Schwarz preconditioner of $A$ (defined through \eqref{prec:B}) and let $B_0$ be the corresponding s.p.d additive Schwarz preconditioner of $A_0$ (defined through \eqref{prec:B0}) and assume $B_0$ is such that
\begin{equation*}
\gamma_0 \bv^{T}B_0 \bv \leq  \bv^{T}A_0 \bv \leq \gamma_1 \bv^{T} B_0 \bv, \quad \forall\, \, \bv \in \mathbb{R}^{n}\;,
 \end{equation*}
for some  $\gamma_0,\gamma_1>0$. Then, the pair $(B,A_0)$ also satisfies {\bf (H0)} with constants $(\beta_0,\beta_1)$:
 \begin{equation}\label{B:coercivity}
\bv^TB \bv \ge \beta_0\; \bv^TA_0 \bv,\quad \text{ for all } \bv \in \mathbb{R}^{n}\;.
\end{equation}
and
\begin{equation}\label{B:boundedness}
\bw^T B \bv \le \beta_1\; \sqrt{\bv^TA_0\bv} \sqrt{\bw^TA_0\bw},\quad \text{ for all }\bv,\;\bw.
\end{equation}
The constants $\beta_0$ and $\beta_1$ are given by
 \begin{equation}\label{betas}
 \beta_0=\frac{c_0^{3}}{c_1^{2}\gamma_1}
  \qquad \beta_1=\frac{c_1^{2}}{c_0\gamma_0}\;.
 \end{equation}
We note that the same results, \eqref{B:coercivity}-\eqref{B:boundedness}, 
hold for $B$ replaced with $B^T$. 
 \end{lemma}
\vskip 0.3cm

With all this relations at hand, we define the s.p.d. matrix
\begin{equation}\label{prec:Z}
Z:=B A_0^{-1}B^T\;.
\end{equation}
that can be used as a preconditioner for $A$ in GMRES. Observe that the actions of $Z^{-1}$ involve actions of both $B^{-1}$ and $B^{-T}$ as well as 
multiplications with $A_0$ (not $A^{-1}_0$). Therefore, the preconditioner $Z$ is computationally feasible. \\

We next prove the main result of the section, which guarantees that the preconditioned GMRES method for $A$ with the s.p.d. preconditioner $Z = B A^{-1}_0 B^T$ will be convergent with bounds depending only on the constants involved in relations between $A$ and $Z$.

\begin{theorem}\label{theorem: convergence for Z}
Let $(A,A_0)$ satisfy assumption {\bf (H0)} with constants $(c_0,c_1)$ and let  $B \in \mathbb{R}^{n\times n}$ be the additive Schwarz preconditioner for $A$, whose inverse is defined through \eqref{prec:B}. Let $Z:=BA^{-1}_0 B^T$ be a preconditioner for $A$. Then, the pair $(A,Z)$ also satisfies {\bf (H0)} with constants $(\alpha_0,\alpha_1)$ defined by:
\begin{equation}\label{def:alpha}
\alpha_0=\frac{c_0}{\beta_1}=\frac{c^{2}_0}{c_1^{2}}\cdot \gamma_0 \qquad \alpha_1=\frac{c_1}{\beta_0}=\frac{c^{2}_1}{c_0^{3} \gamma_1}\;.
\end{equation}
Furthermore, the preconditioned GMRES method for $A$ with the s.p.d. preconditioner
$Z$ converges  with bounds:
\begin{equation}\label{estimate:0}
\|\br_m\|_{Z^{-1}} =
\|{\overline \br}_m\|_{Z} \leq \left(1-\left (\frac{\alpha_0}{\alpha_1} \right )^2 \right)^{\left (\frac{m}{2} \right )} 
\|{\overline \br}_0\|_{Z}=
\left(1- \left (\frac{\alpha_0}{\alpha_1} \right )^2 \right)^{\left (\frac{m}{2} \right )} 
\|\br_0\|_{Z^{-1}}\,,
\end{equation}
where  ${\overline \br}_m = Z^{-1} \br_m = Z^{-1} (\bb - A \bx_m)$ is the preconditioned 
residual at the $m$-th iteration with ${\overline \br}_0= Z^{-1} \br_0 :=
Z^{-1}(\bb-A \bx_0)$; $\|\cdot\|_{Z}$ and $\|\cdot\|_{Z^{-1}}$ are the inner-product norms induced by 
the s.p.d matrices $Z$ and $Z^{-1}$, respectively. 

If we use GMRES with right preconditioning in the standard norm, i.e. $\|\bv\|=  \sqrt{\bv^T \bv}$, the following estimate holds
\begin{equation}\label{estimate-right} 
\|\br_m\| \le \left (\|Z\|\|Z^{-1}\| \right )^{\frac{1}{2}}\; 
\left(1- \left (\frac{\alpha_0}{\alpha_1} \right )^2 \right)^{\left (\frac{m}{2} \right )} 
\|\br_0\|.
\end{equation}
Note that in our case $\text{cond}(Z) = \|Z\|\|Z^{-1}\|\simeq \text{cond}(A_0) =  \O(h^{-2})$.
\end{theorem}
\begin{proof}
From Lemma \ref{le:1}, we know that $(B,A_0)$ satisfy {\bf (H0)} with $(\beta_0,\beta_1)$. In particular, the relations \eqref{B:coercivity}--\eqref{B:boundedness} (used for $B^T$) show that $X := A^{-\frac{1}{2}}_0 B^T A^{-\frac{1}{2}}_0$ is well--conditioned.
More precisely, we have
\begin{equation*}
\beta_0\; \|\bv \|^{2} \le \|X \bv \|^{2} \le \beta_1\; \|\bv\|^{2} \quad \text{ for all } \bv\in \mathbb{R}^{n}\;.
\end{equation*}
That is, the s.p.d. matrix $X^TX$ is well--conditioned. 
The coercivity of $A$ in terms of $A_0$ expressed in \eqref{A:coercivity}, and $X^TX$ being well--conditioned (or, bounded) imply that 
$A^{-\frac{1}{2}}_0AA^{-\frac{1}{2}}_0$ is coercive also in terms of $X^T X$:
\begin{equation*}
\bv^{T} A^{-\frac{1}{2}}_0AA^{-\frac{1}{2}}_0\bv \ge c_0 \|\bv\|^{2} \geq \frac{c_0}{\beta_1}  \bv^{T} \, X ^{T}X \bv \quad \text{ for all } \bv\in \mathbb{R}^{n}\;.
\end{equation*}
Hence, $A$ is coercive in terms of 
$A^{\frac{1}{2}}_0 X^TX A^{\frac{1}{2}}_0 = BA^{-1}_0 B^T = Z$, which is the first desired result. 

Similarly, the boundedness of $A$ in terms of $A_0$, expressed in \eqref{A:boundedness}, and $X^TX$ being well--conditioned (or coercive)  imply that 
$A^{-\frac{1}{2}}_0AA^{-\frac{1}{2}}_0$ is bounded also in terms of $X^T X$:
\begin{equation*}
\bw^{T} A^{-\frac{1}{2}}_0AA^{-\frac{1}{2}}_0\bv  \leq  c_1 \|\bw\| \|\bv\| \leq  \frac{c_1}{\beta_0} \sqrt{ \bw^{T} X ^{T}X \bw} \, \sqrt{\bv^T X ^{T}X \bv}, \quad \text{ for all } \bv,\bw\, \in \mathbb{R}^{n}\;,
\end{equation*}
 which is equivalent to say that
$A$ is bounded in terms of $A^{\frac{1}{2}}_0 X^TX A^{\frac{1}{2}}_0 = BA^{-1}_0 B^T=Z$. This completes the 
proof that the pair $(A, Z)$ verifies assumption {\bf (H0)} with constants $(\alpha_0,\alpha_1)$ as defined in  \eqref{def:alpha}.\\

Since our preconditioner $Z$ is s.p.d., we can apply the (unpreconditioned) GMRES
 in the standard vector norm to the symmetrically transformed problem
${\overline A} {\overline \bx} = {\overline \bb}$, where
${\overline A} = Z^{-\frac{1}{2}} A Z^{-\frac{1}{2}}$, ${\overline \bb} = Z^{-\frac{1}{2}} \bb$ and 
${\overline \bx} = Z^{\frac{1}{2}} \bx$. After rewriting the resulting algorithm in the original quantities
it is clear that the resulting method will minimize the (true) residual in $\|.\|_{Z^{-1}}$-norm.
A standard application of the GMRES convergence theory \cite{elman} gives \eqref{estimate:0}.

If we were to apply GMRES to $A \bx = \bb$ using e.g., right preconditioning (cf., \cite{SaadBook03}, \S~3.2.2)),
 then estimate \eqref{estimate-right} is easily seen since
\begin{equation*}
\begin{array}{rl}
\|\br_m\| &=\min\;\{\|p_m(AZ^{-1})\br_0\|:\; p_m(0) =1,\;p_m \text{ polynomial of degree }\le m\}\\
&\le \|Z^{\frac{1}{2}}\|\; \min\;\{\|Z^{-\frac{1}{2}}p_m(AZ^{-1})\br_0\|:\; p_m(0) =1,\;p_m \text{ polynomial of degree }\le m\}\\
& = \|Z^{\frac{1}{2}}\| \; \min\;\{\|p_m(Z^{-\frac{1}{2}}AZ^{-\frac{1}{2}})(Z^{-\frac{1}{2}}\br_0)\|:\; p_m(0) =1,\;p_m \text{ polynomial of degree }\le m\}\\
& \le 
\|Z^{\frac{1}{2}}\| \;
\left(1- \left (\frac{\alpha_0}{\alpha_1} \right )^2 \right)^{\left (\frac{m}{2} \right )} 
\|Z^{-\frac{1}{2}}\br_0\| \\
& \le \|Z^{\frac{1}{2}}\| \;\|Z^{-\frac{1}{2}}\| \;
\left(1- \left (\frac{\alpha_0}{\alpha_1} \right )^2 \right)^{\left (\frac{m}{2} \right )} 
\|\br_0\|.
\end{array}
\end{equation*}
Thus the proof of the theorem is complete.

\end{proof}

\subsection{ Another auxiliary s.p.d. preconditioner $W$}

To close the section we define another  s.p.d. preconditioner, that we shall  denote by $W$,  and that uses also only actions of $B^{-1}$ and $B^{-T}$:  
\begin{equation}\label{prec:W}
W^{-1} =\frac{1}{2}(B^{-1} + B^{-T}).
\end{equation}

Observe that  estimates \eqref{B:coercivity} and \eqref{B:boundedness} from Lemma \ref{le:1}, show that the pairs $(B,A_0)$ and $(B^{T},A_0)$  satisfy {\bf (H0)} with constants $(\beta_0, \beta_1)$ as given in \eqref{betas}. Therefore, both pairs satisfy the assumptions of Lemma \ref{le:0} from which we deduce that the pairs $(B^{-1}, A^{-1}_0)$  and $(B^{-T}, A^{-T}_0)$ are also related through {\bf (H0)} with constants $(\frac{\beta_0}{\beta^{2}_1}, \beta_0^{-1})$. 
This implies in particular, that the pair $(W, A_0)$ exhibits property {\bf (H0)} which means that the preconditioner $W$ is s.p.d. and is spectrally equivalent to $A_0$. 

Arguing then as in the second part of the proof of Theorem \ref{theorem: convergence for Z} but with $W$ in place  of  $Z$, one can guarantee  that the GMRES method applied to $A$ with $W$ as an s.p.d.  preconditioner will be convergent. We omit the details for brevity.

\section{A combined preconditioner}\label{section: combined preconditioner}\label{sec:4}
In this section, we introduce another preconditioner which  in a sense combines the best of both preconditioners $B$ and $Z = B A^{-1}_0B^T$. We define its inverse $\mathcal{B}^{-1}$ by forming the linear combination
\begin{equation*}
\B^{-1} = B^{-1} +\sigma Z^{-1}.
\end{equation*}
The parameter $\sigma \in {\mathbb R}$ is allowed to change from iteration to iteration inside the GMRES iterative solver. Therefore, $\B$ can be regarded as a  variable--step (or flexible) preconditioner. \\

Observe that for $\sigma \ge 0$, by virtue of the analysis of the previous section, the pair $(\B^{-1},A^{-1}_0)$ verifies assumption {\bf (H0)}; i.e, $\B^{-1}$ is coercive and bounded in the $A^{-1}_0$ norm. 
We  now describe the (practical) construction of the variable--step preconditioner $\B^{-1}$, but considering a more general form:
\begin{equation}\label{b22}
\B^{-1}=\alpha B^{-1} +\sigma Z^{-1}
\end{equation}
without assuming  the coefficients $\alpha$ and $\sigma$ to have nonnegative sign. 
We note that since the coefficients $\alpha$ and $\sigma$ will be allowed  to vary from one step to the following, and so  will do $\B^{-1}$, certainly $\B^{-1}$ depend on the interation $m$ However, to avoid  complicated notation and to ease the presentation, we do not make explicit reference to the iteration $m$ in the notation of  $\B^{-1}$ and $\alpha$ and $\sigma$, since at all times it will be clear and no confusion might arise.

$\mathcal{B}$ depends on the parameters $\alpha$ and $\sigma$ and varies from one step to the next. We suggest writing $\mathcal{B}_{m}$ instead to show the explicit dependence on the iteration count, and the
same for the parameters  $\alpha$ and $\sigma$ that define $\mathcal{B}$

We first discuss the construction of the combined preconditioner and  provide a first  convergence result for $\B^{-1}$ which asserts faster convergence within GMRES than the one obtained with the preconditioner $Z^{-1}$ at a fixed iteration. 
In Section \ref{sec:42} we present a variant of the combined preconditioner for which one can indeed guarantee convergence with respect to  the GMRES method using only the $Z^{-1}$  (at the whole sequence of iterates rather that at a fixed one). 
\subsection{Construction of the variable preconditioner}\label{sec:41}

We consider the linear system of equations  \eqref{system:0} that we solve by the preconditioned GMRES method with preconditioner $\B^{-1}$ as defined in \eqref{b22}. We now explain how the coefficients are $\alpha$ and $\sigma$ set inside the GMRES iteration. 
Let $\| \cdot \| = \sqrt{(\cdot,\cdot)}$ and $\|\cdot\|_* = \sqrt{(.,.)_*}$ be two inner product norms, to be specified and chosen later on, and whose role will become clear in the process.\\

For $m\ge 0$, we denote by  $\bx_m$ the  $m^{th}$-iterate and by $\br_m = \bb - A \bx_m$ the residual.  At the  $(m+1)^{th}$ iteration of GMRES, we construct the new
search direction $\bd_{m+1}$  based not only on the previous search directions $\{\bd_j\}^m_{j=0}$ but also on the  two preconditioned residuals $B^{-1} \br_m$ and $Z^{-1} \br_m$, as follows:
\begin{equation}\label{eq:des:bes}
\beta_{m+1}\bd_{m+1} = \beta_{m+\frac{1}{3}} B^{-1} \br_m + \beta_{m+\frac{2}{3}} Z^{-1} \br_m
+ \sum\limits^m_{j=0}\beta_j \bd_j.
\end{equation}
Here,  the coefficients $\beta_j$, $j=0,1,\;\dots,\;m$
 are chosen such that 
 \begin{equation*}
 (\bd_{m+1},\bd_j)_* = 0 \quad \mbox{ for} \quad j<m+1 \qquad  \mbox{ and }\quad
\|\bd_{m+1}\|_* = \sqrt{(\bd_{m+1},\;\bd_{m+1})_*} = 1.
\end{equation*}
It is clear then, that the coefficients 
$\beta_{m+\frac{s}{3}}$, $s=1,2$ can be considered arbitrary parameters at this point. 
For any such fixed pair in GMRES, the next iterate $\bx_{m+1}$  is then computed by minimizing the residual:
\begin{equation*}
\|\bb - A \bx_{m+1}\| = \|\bb - A (\bx_m + \sum\limits^{m+1}_{j=0}\alpha_j \bd_j)\| \mapsto \min
\end{equation*}
with respect to the coefficients $\{\alpha_j\}^{m+1}_{j=0}$. Notice  that out of the 
two coefficients $\beta_{m+\frac{s}{3}}$, $s=1,2$, only their ratio 
\begin{equation*}
\sigma = \sigma_{m+1} \equiv \frac{\beta_{m+\frac{2}{3}}}
{\beta_{m+\frac{1}{3}}},
\end{equation*}
 can be considered as a free  parameter (the rest is compensated by the 
$\alpha_{m+1}$-coefficient). \\

In practice, we proceed as follows. At step $m+1$, based on the previous search directions
$\{\bd_j\}^m_{j=0}$ and the preconditioned residuals $B^{-1} \br_m$ and $Z^{-1} \br_m$, we have to solve the 
minimization problem: 
\begin{equation}\label{eq:min:2}
\|\bb - A (\bx_m + \sum\limits^m_{j=0}\alpha_j \bd_j + \alpha_{m+\frac{1}{3}}B^{-1} \br_m +
\alpha_{m+\frac{2}{3}}Z^{-1} \br_m)\| \mapsto \min\;,
\end{equation}
with respect to the coefficients $\{\alpha_j\}^m_{j=0}$, and $\alpha_{m+\frac{s}{3}}$, $s=1,2$. 
As we show next the solution of such minimization problem can be reduced to the solution of a  two-by-two system, by choosing appropriately the inner product $(\cdot ,\cdot )_*$. 

Consider the quadratic functional $\J(\balpha)$ (as a function of the coefficients $\balpha = (\alpha_r)$)
\begin{equation*}
\J(\balpha)\equiv\|\br_m - A( \sum\limits^m_{j=0}\alpha_j \bd_j + \alpha_{m+\frac{1}{3}}B^{-1} \br_m
+\alpha_{m+\frac{2}{3}}Z^{-1} \br_m)\|^2\;.
\end{equation*}
Then it is obvious that the minimization problem \eqref{eq:min:2} reduces to minimize the functional with respect to the coefficients $\balpha = (\alpha_r)$. We set now the inner product  $(\cdot,\cdot)_* = (A(\cdot),.A(\cdot))$, which is equivalent to assume that the search directions $\{\bd_j\}$ are $(A(.),A(.))$ orthogonal.
Then, we have 
\begin{equation*}
0= \frac{1}{2}\frac{\partial \J}{\partial \alpha_j} =
\alpha_j - (\br_m - \alpha_{m+\frac{1}{3}}A B^{-1} \br_m - \alpha_{m+\frac{2}{3}}AZ^{-1} \br_m,\;A\bd_j),\quad
\; j \le m,
\end{equation*}
which gives
\begin{equation}\label{equations for alpha_j}
\alpha_j = (\br_m,\;A\bd_j) - \alpha_{m+\frac{1}{3}}(A B^{-1} \br_m,\;A\bd_j)
 - \alpha_{m+\frac{2}{3}}(AZ^{-1} \br_m,\;A\bd_j),
\; j \le m.
\end{equation}
Setting now the partial derivatives of $\J$ w.r.t. $\alpha_{m+\frac{s}{3}}$ to zero, we get
\begin{equation}\label{equations for fractional alphas}
\begin{array}{rl}
\alpha_{m+\frac{1}{3}}\|A B^{-1} \br_m\|^2 - 
(\br_m - A( \sum\limits^m_{j=0}\alpha_j \bd_j 
+\alpha_{m+\frac{2}{3}}Z^{-1} \br_m,\; A B^{-1} \br_m) &=0,\\
\alpha_{m+\frac{2}{3}}\|A Z^{-1} \br_m\|^2 - 
(\br_m - A( \sum\limits^m_{j=0}\alpha_j \bd_j + \alpha_{m+\frac{1}{3}}B^{-1} \br_m),\; A Z^{-1} \br_m) &=0.
\end{array}
\end{equation}
Substituting $\alpha_j$ from \eqref{equations for alpha_j} into \eqref{equations for fractional alphas}
we  end up with a system of two equations where the only two unknowns are the coefficients
$\alpha_{m+\frac{s}{3}}$ with $s=1,2$:
\begin{equation}\label{particular system for fractional alphas}
\begin{array}{l}
\bullet \quad \alpha_{m+\frac{1}{3}}\left (\|AB^{-1}\br_m\|^2-\sum\limits_j (AB^{-1}\br_m,\;A\bd_j)^2
\right )\\
\qquad \qquad \quad+  \alpha_{m+\frac{2}{3}}\left ((AZ^{-1}\br_m,\;AB^{-1}\br_m)-\sum\limits_j (A Z^{-1}\br_m,\; A\bd_j)(A B^{-1}\br_m,\; A\bd_j)\right )\\
= (\br_m,\;AB^{-1}\br_m)-\sum\limits_j
(\br_m,\;A\bd_j) (AB^{-1}\br_m,\;A\bd_j)
= (\br_m,\; AB^{-1}\br_m -\sum\limits_j (AB^{-1}\br_m,\;A\bd_j)A\bd_j),\\
\bullet \quad\alpha_{m+\frac{1}{3}}\left ((AZ^{-1}\br_m,\;AB^{-1}\br_m)-\sum\limits_j (A B^{-1}\br_m,\; A\bd_j)(A Z^{-1}\br_m,\; A\bd_j)\right ) \\
\qquad \qquad \quad+ \alpha_{m+\frac{2}{3}}\left (\|AZ^{-1}\br_m\|^2-\sum\limits_j (AZ^{-1}\br_m,\;A\bd_j)^2
\right )\\
= (\br_m,\;AZ^{-1}\br_m)-\sum\limits_j
(\br_m,\;A\bd_j) (AZ^{-1}\br_m,\;A\bd_j)= (\br_m,\; AZ^{-1}\br_m -\sum\limits_j (AZ^{-1}\br_m,\;A\bd_j)A\bd_j).
\end{array}
\end{equation}
To show the solvability of the above system for $\alpha_{m+\frac{s}{3}}$ with $s=1,2$ (which will imply that the variable--step preconditioner $\B^{-1}$ is well defined), we use the following lemma.
\begin{lemma}\label{lemma: for projections}
Let $(H,\;(.,.))$ be a Hilbert space with inner product $(\cdot,\cdot)$ and let  $\bh,\;\bbf,\;\bg \in H$.
Let $S$ be a finite dimensional subspace of $H$ spanned by an orthonormal system $\{\bp_j\}^m_{j=1}$, i.e.,
$(\bp_i,\;\bp_j) = \delta_{i,j}$. Let $\pi=\pi_S: H\longrightarrow S$ be the orthogonal projection on $S$, with respect to the inner product $(\cdot,\cdot)$. Then, the best approximation to $\bh$ from elements from $S$ augmented by the two vectors $\bbf$ and $\bg$ is given as the
solution of the least squares (or minimization) problem
\begin{equation}\label{abstract least-squares problem}
\min_{
\substack{
\alpha_r\\
 r=\frac{1}{3},\frac{2}{3},\;1,\dots,\;m}}  \|\bh - \alpha_{\frac{1}{3}} \bbf - \alpha_{\frac{2}{3}} \bg - \sum\limits_j \alpha_j \bp_j\| \mapsto 
\min
\end{equation}
over the coefficients $\{\alpha_r\}$, $r=\frac{1}{3},\frac{2}{3},\;1,\dots,\;m$. Solving problem \eqref{abstract least-squares problem} is equivalent to solve the following two-by-two system
\begin{equation}\label{Gram system}
\left( \begin{array}{cccc}
&\|(I-\pi)\bbf\|^2 &\quad ((I-\pi)\bbf,\;(I-\pi)\bg) \\
&((I-\pi)\bbf,\;(I-\pi)\bg) & \|(I-\pi)\bg\|^2 
\end{array}\right)
\cdot \left[\begin{array}{cc}
\alpha_{\frac{1}{3}}\\
\alpha_{\frac{2}{3}}
\end{array}\right] = \left[\begin{array}{cc}
(\bh,\;(I-\pi)\bbf)\\
(\bh,\;(I-\pi)\bg)
\end{array}\right],
\end{equation}
which has a unique solution provided $(I-\pi)\bbf$ and $(I-\pi)\bg$ are linearly independent.
If $(I-\pi)\bbf$ and $(I-\pi)\bg$ are  linearly dependent, there is also a solution,
since the r.h.s. in \eqref{Gram system} is compatible.  More specifically, the component
\begin{equation*}
\bh^\perp = \alpha_{\frac{1}{3}} (I-\pi)\bbf + \alpha_{\frac{2}{3}}(I-\pi)\bg,
\end{equation*}
is unique including the case of linear dependent components $(I-\pi)\bbf$ and $(I-\pi)\bg$ 
(even when $(I-\pi)\bbf = (I-\pi)\bg = 0$.) \\
The remaining coefficients $\{\alpha_j\}$ are computed  from 
$\pi (\bh - \alpha_{\frac{1}{3}}\bbf - \alpha_{\frac{2}{3}}\bg) = \sum\limits_j \alpha_j \bp_j$, that is:
\begin{equation*}
\alpha_j = (\bh - \alpha_{\frac{1}{3}}\bbf - \alpha_{\frac{2}{3}}\bg, \;\bp_j).
\end{equation*}
\end{lemma}
\begin{proof}
It is clear that the least-squares problem \eqref{abstract least-squares problem} reduces to finding 
the best approximation to $\bh$ from the space spanned by the two vectors
$(I-\pi)\bbf$ and $(I-\pi)\bg$. Indeed, we can rewrite \eqref{abstract least-squares problem} as 
\begin{equation*}
\|(I-\pi)\bh -\alpha_{\frac{1}{3}} (I-\pi)\bbf - \alpha_{\frac{2}{3}} (I-\pi)\bg - \sum\limits_j \alpha{'}_j\bp_j\|
\mapsto \min.
\end{equation*}
Since the last component $\sum\limits_j \alpha{'}_j\bp_j$ is orthogonal to $(I-\pi)\left (\bh - \alpha_{\frac{1}{3}} 
\bbf - \alpha_{\frac{2}{3}} \bg\right )$, the above minimum equals
\begin{equation*}
\begin{array}{c}
\min\limits_{\alpha_{\frac{1}{3}},\;\alpha_{\frac{2}{3}}}\;\min\limits_{\alpha^{'}_j}\;
\|(I-\pi)\bh -\alpha_{\frac{1}{3}} (I-\pi)\bbf - \alpha_{\frac{2}{3}} (I-\pi)\bg - \sum\limits_j \alpha{'}_j\bp_j\|\\
= \min\limits_{\alpha_{\frac{1}{3}},\;\alpha_{\frac{2}{3}}}\;\min\limits_{\alpha^{'}_j}\;
 \left (
\|(I-\pi)\bh -\alpha_{\frac{1}{3}} (I-\pi)\bbf - \alpha_{\frac{2}{3}} (I-\pi)\bg\|^2 
+ \|\sum\limits_j \alpha{'}_j\bp_j\|^2 \right )^{\frac{1}{2}}\\
= \min\limits_{\alpha_{\frac{1}{3}},\;\alpha_{\frac{2}{3}}}\;
\|(I-\pi)\bh -\alpha_{\frac{1}{3}} (I-\pi)\bbf - \alpha_{\frac{2}{3}} (I-\pi)\bg \|\\
= \min\limits_{\alpha_{\frac{1}{3}},\;\alpha_{\frac{2}{3}}}\;
\left (\|\bh -\alpha_{\frac{1}{3}} (I-\pi)\bbf - \alpha_{\frac{2}{3}} (I-\pi)\bg \|^2-\|\pi \bh\|^2 
\right )^{\frac{1}{2}}.
\end{array}
\end{equation*}
The last problem leads exactly to the Gram system \eqref{Gram system}.   We note that
\begin{equation*}
\min\limits_{\alpha_{\frac{1}{3}},\;\alpha_{\frac{2}{3}}}\;
\|\bh -\alpha_{\frac{1}{3}} (I-\pi)\bbf - \alpha_{\frac{2}{3}} (I-\pi)\bg \|
 = \|\bh - \bh^\perp\|, 
\end{equation*}
where $\bh^\perp$ is the unique best approximation to $\bh$ from the space  spanned by the vectors 
$(I-\pi)\bbf$ and $(I-\pi)\bg$. The latter space has  dimension zero, one, or two.
This completes the proof.
\end{proof}
We now apply last Lemma to our case, to show that the system \eqref{particular system for fractional alphas} has a solution and hence $\B^{-1}$ is well defined.
We set  $\bh = A^{-1}\br_m$, $\bbf = B^{-1}\br_m$, $\bg = Z^{-1} \br_m$
and $\{\bp_j\} =\{\bd_j\}^m_{j=0}$ 
 for the vector space with inner-product $(\cdot,\;\cdot)_* = (A(\cdot),\;A (\cdot))$. Using then that the 
$\{\bd_j\}$ are $(\cdot,\;\cdot)_*$-orthonormal, we conclude by applying Lemma \ref{lemma: for projections} that  the system 
\eqref{particular system for fractional alphas} is solvable.\\

Once the coefficients $\alpha_{m+\frac{s}{3}}$, $s=1,2$ are determined, we compute the new direction 
$\bd_{m+1}$ from
\begin{equation*}
\beta_{m+1}\bd_{m+1} = \alpha_{m+\frac{1}{3}} B^{-1} \br_m + \alpha_{m+\frac{2}{3}} Z^{-1} \br_m
-\sum\limits^m_{j=0}\beta_j \bd_j,
\end{equation*}
by choosing the coefficients $\{\beta_j\}^m_{j=0}$ to satisfy 
the required orthogonality conditions $$
(\bd_{m+1},\bd_j)_*=(A\bd_{m+1},A\bd_j)=0 \quad \mbox{  for   }\quad j<m+1,
$$
which gives
(assuming by induction that $(\bd_j,\;\bd_k)_*= \delta_{j,k}$)
\begin{equation*}
\beta_j = (\alpha_{m+\frac{1}{3}} B^{-1} \br_m + \alpha_{m+\frac{2}{3}} Z^{-1} \br_m,\; \bd_j)_*
\text{ for } j \le m.
\end{equation*}
The last coefficient, $\beta_{m+1}$, is computed so that $\|\bd_{m+1}\|_*= 1$.

\subsection{Convergence}
We now establish two results that provide estimates for the convergence of the variable preconditioned GMRES method. 
\begin{theorem}\label{theorem:2}
Let $(A,A_0)$ satisfy assumption {\bf (H0)} with constants $(c_0,c_1)$ and let  $B \in \mathbb{R}^{n\times n}$ be the additive Schwarz preconditioner for $A$, whose inverse is defined through \eqref{prec:B}. Let $Z:=BA^{-1}_0 B^T$ be a preconditioner for $A$.  Let $\B$ be the variable--step preconditioner with inverse defined through \eqref{b22} with coefficients determined inside the GMRES iteration by minimization of the residual. Then, the variable preconditioned GMRES method for $A$, at every step $k$, converges at least as fast as performing one iteration of the restarted (at step $k$) preconditioned GMRES method with preconditioner $Z$. 
\end{theorem}
\begin{proof}
The proof of the Theorem follows by the definition of $\B^{-1}$ (as explained before). From its construction it is straightforward to infer the following comparative convergence estimate
\begin{equation*}
\|\br_{m+1}\| \le \min_{\alpha, \;\sigma}\;\|\br_m - A(\alpha B^{-1} + \sigma Z^{-1}) \br_m\|
\le \min_{\sigma}\; \|\br_m - \sigma A Z^{-1} \br_m\|.
\end{equation*}
Moreover, if we choose  $\|\cdot\|$ to be the norm $\|\bv\|_{Z^{-1}} = \sqrt{\bv^T Z^{-1} \bv}$, we can apply the result from Theorem  \ref{theorem: convergence for Z}. We then have
\begin{equation*}
\begin{array}{rl}
\|\br_{m+1}\|_{Z^{-1}} \displaystyle &\le \min_{\sigma}\; \|\br_m - \sigma A Z^{-1} \br_m\|_{Z^{-1}}\\
\displaystyle & \le \left [1-\left (\frac{\alpha_0}{\alpha_1} \right )^2 \right ]^{\frac{1}{2}} \|\br_m\|_{Z^{-1}} \le \dots \le
\left [1-\left (\frac{\alpha_0}{\alpha_1} \right )^2 \right ]^{\frac{m+1}{2}} \|\br_0\|_{Z^{-1}}.
\end{array}
\end{equation*}
That is, estimate \eqref{estimate:0}  is an upper bound for the  convergence estimate of  the combined preconditioned GMRES method. Therefore, we showed that the rate of convergence of the combined preconditioned GMRES can be bounded at least with the same bound as of  the GMRES method with preconditioner $Z$ only. \end{proof}

We wish to stress that Theorem \ref{theorem:2} ensures that at {\it every iteration} the combined preconditioner gives better reduction of the residual than applying one iteration of the restarted GMRES using only one of the preconditioners, $Z$ (or $B$), and dropping the previous search directions.
This is sufficient to show convergence of the combined preconditioned GMRES algorithm since in our case with $Z$ being s.p.d., we can prove that it is convergent (cf. Theorem \ref{theorem: convergence for Z}).\\

 We close the section by showing (arguing as before) convergence for another combined preconditioner, which performance is also 
demonstrated in the numerical experiments section~\ref{sec:5}. 
\begin{theorem}\label{theorem:4}
Let $(A,A_0)$ satisfy assumption {\bf (H0)} with constants $(c_0,c_1)$ and let  $B \in \mathbb{R}^{n\times n}$ be the additive Schwarz preconditioner for $A$, whose inverse is defined through \eqref{prec:B}. Let $\B$ be the variable--step preconditioner with inverse defined through:
\begin{equation}\label{bbt}
\B^{-1}= \alpha B^{-1} +\sigma B^{-T}
\end{equation}
 with coefficients determined inside the GMRES iteration by minimization of the residual. Then, the variable preconditioned GMRES method for $A$, at every step $k$, converges at least as fast as performing one iteration of the restarted (at step $k$) preconditioned GMRES method with preconditioner $W$ defined in \eqref{prec:W}.
\end{theorem}

\begin{proof}
We start recalling that the GMRES method with the s.p.d. preconditioner $W$ defined in \eqref{prec:W} is convergent for $A$. Then, arguing  as in the proof of Theorem \ref{theorem:2},  we have
\begin{equation*}
\begin{array}{rl}
\|\br_{m+1}\| &\le \min_{\alpha, \;\sigma}\;\|\br_m - A(\alpha B^{-1} + \sigma B^{-T}) \br_m\|\\
& \le \min_{\sigma}\; \|\br_m - \sigma A (\frac{1}{2}\left (B^{-1}+B^{-T} \right )) \br_m\| \\
& = \min_{\sigma}\; \|\br_m - \sigma A W^{-1} \br_m\|.
\end{array}
\end{equation*}
Since for the latter one we have convergence the proof is complete.
\end{proof}

Notice that, arguing as in the proof of Theorem \ref{theorem:2} , one could straightforwardly establish the bounds:
\begin{equation*}
\begin{aligned}
\|\br_{m+1}\| &\le \min_{\alpha, \;\sigma}\;\|\br_m - A(\alpha B^{-1} + \sigma B^{-T}) \br_m\|
 \le  \min_{\sigma}\; \|\br_m - \sigma A B^{-T} \br_m\|. &&\\
 \|\br_{m+1}\| &\le \min_{\alpha, \;\sigma}\;\|\br_m - A(\alpha B^{-1} + \sigma B^{-T}) \br_m\|
 \le  \min_{\alpha}\; \|\br_m - \alpha A B^{-1} \br_m\|. &&
\end{aligned}
\end{equation*}
However, they might or might not be overestimates. From them, it is not possible to deduce a rigorous convergence comparison.

\subsection{ A Modified GMRES method with combined preconditioning}\label{sec:4} 

Here, we present a modified GMRES method which can be guaranteed to converge better than standard GMRES using preconditioner $Z$ but which is much more expensive than the previous method described  so far.\\

The possible modification of the algorithm we present now,  allows for comparing the convergence of  the GMRES method using the combined preconditioner and  using only one of the preconditioners. \\

The idea is to incorporate in \eqref{eq:des:bes} the set of search directions $\{\bd^Z_j\}^m_{j=0}$  generated by running separately  the GMRES method with the $Z$ preconditioner. Obviously we are increasing the dimension of the Krylov subspaces over which the minimization  of the residual norm is done.
Notice that the variable--step preconditioner \eqref{b22} corresponds to the case $m_0=0$ in the present algorithm.

We start constructing two sets of search directions, the set $\{\bd_j\}^m_{j=1}$ and another set $\{\bd^Z_j\}^m_{j=0}$ that corresponds to running GMRES using $Z$ 
as a preconditioner which also produces its residual $\br^Z_m$.
Based on $\br^Z_m$ and $\{\bd^Z_j\}^m_{j=0}$, we compute $\bd^Z_{m+1}$ which is a linear combination of
$Z^{-1}\br^Z_m$ and the set $\{\bd^Z_j\}^m_{j=1}$. This process is in principle independent of the combined preconditioned GMRES and it can be run separately.

 At step $m+1$ of the combined method, we construct a new search direction $\bd_{m+1}$ as follows
\begin{equation}\label{eq:des:bes2}
\beta_{m+1}\bd_{m+1} = \beta_{m+\frac{1}{3}}B^{-1}\br_m + \beta_{m+\frac{2}{3}} Z^{-1} \br_m + \sum\limits^m_{j=0} \beta_j \bd_j
+\sum\limits^m_{j=0} \beta^Z_j \bd^Z_j.
\end{equation}
The difference from the previous version (see \eqref{eq:des:bes}) is that we also use the search directions coming from the 
GMRES process corresponding to the preconditioner $Z$.

The coefficients $\beta_j$ and $\beta^Z_j$ are computed from the orthogonality conditions:
\begin{equation}\label{new search dieraction}
(\bd_{m+1}, \bd_j)_* = 0,\quad \mbox{and } \quad  (\bd_{m+1},\;\bd^Z_j)_*= 0, \quad \mbox{  for  }\quad \; j=0,\;\dots,\;m.
\end{equation}
Note that by construction, we are ensuring that
\begin{equation}\label{additional orthogonality conditions}
(\bd_j,\;\bd^Z_k)_* =0 \text{ for } j>k \ge 0.
\end{equation}
Now, the coefficients $\beta_{\frac{1}{3}}$ and $\beta_{\frac{2}{3}}$  in \eqref{eq:des:bes2} can be computed as in Lemma \ref{lemma: for projections} where now $\pi$ is the $(.,.)_*$-orthogonal projection on the space spanned by the two sets of search directions, $\{\bd_j\}^m_{j=1}$ and  $\{\bd^Z_j\}^m_{j=0}$, and setting ${\bf h}=\br_m$. Then, the system 
for 
the coefficients $\beta_j$ and $\beta^Z_j$, $j \le m$, \eqref{new search dieraction} is solvable and it simplifies somewhat using the orthogonality of the two search directions given in  \eqref{new search dieraction} and  \eqref{additional orthogonality conditions}. 
The next iterate of the method takes the form
\begin{equation*}
\begin{array}{rl}
\bx_{m+1} &= \bx_m + \sum\limits^{m+1}_{j=0}
\left ({\overline \alpha}_j \bd_j + {\overline \alpha}^Z_j \bd^Z_j \right )\\
&= \bx_0 +\sum\limits^{m+1}_{j=0}
\left (\beta_j \bd_j + \beta^Z_j \bd^Z_j \right ).
\end{array}
\end{equation*}
The latter equality holds, since by construction, $\bx_m-\bx_0$ is spanned by the previous search directions 
$\{\bd_j\}^m_{j=0}$ and $\{\bd^Z_j\}^m_{j=0}$. 
The coefficients $\{\beta_j\}^m_{j=0}$ and
$\{\beta^Z_j\}^m_{j=0}$ are computed by solving the minimization problem
\begin{equation*}
\begin{array}{rl}
\|\br_{m+1}\| &= \min\limits_{\{{\overline \alpha}_j\}^m_{j=0},\;\{{\overline \alpha}^Z_j\}^{m+1}_{j=0}}\; \|\br_m -
A\left (\sum\limits^{m+1}_{j=0}{\overline \alpha}_j \bd_j + \sum\limits^{m+1}_{j=0} {\overline \alpha}^Z_j \bd^Z_j\right )\|\\
& = \min\limits_{\{\beta_j\}^m_{j=0},\;\{\beta^Z_j\}^{m+1}_{j=0}}\; \|\br_0 -
A\left (\sum\limits^{m+1}_{j=0}\beta_j \bd_j + \sum\limits^{m+1}_{j=0} \beta^Z_j \bd^Z_j\right )\|.
\end{array}
\end{equation*}
This immediately shows that
\begin{equation*}
\begin{array}{rl}
\|\br_{m+1}\| &= \min\limits_{\{\beta_j,\;\beta^Z_j\}^{m+1}_{j=0}}\; \|\br_0 -
A\left (\sum\limits^{m+1}_{j=0}\beta_j \bd_j + \sum\limits^{m+1}_{j=0} \beta^Z_j \bd^Z_j\right )\| \\
& \le \min\limits_{\{\beta^Z_j\}^{m+1}_{j=0},}\; \|\br_0 -A(\sum\limits^{m+1}_{j=0} \beta^Z_j \bd^Z_j)\| \\
& = \|\br^Z_{m+1}\|.
\end{array}
\end{equation*}

At this point, we wish to note  that the full version of the modified method might not be very practical for large $m$ since we need $2m$ search directions and require $2m$ applications of the preconditioner $Z$ (in addition to $m$ actions of the preconditioner $B$). 
A more practical version of the method might be to use restart for computing the $Z$-based search directions (or both sets of search directions) and use the above modified method for $m=m_0 \ge 0$ steps. We have focused in the present paper on the (extreme) case $m_0 =0$, corresponding to the variable--step preconditioner \eqref{b22}.
The more general case $m_0 >0$  however requires more detailed investigation and is left for a future study.

\section{Applications and numerical results}\label{sec:5}

In this section we present an application of the results  presented in the previous sections, that will allow us to verify the developed theory and assess the performance of the different preconditioners.\\

The application we consider comes from a nonsymmetric Discontinuous Galerkin discretization of an elliptic problem, which was the starting motivation of this work. In \cite{paola-blanca1,paola-blanca2}, additive and multiplicative Schwarz preconditioners were developed for the solution of the above algebraic system. In both works, the authors show that the GMRES convergence theory cannot be applied for explaining the observed convergence since the preconditioned system does not satisfy the {\it sufficient} conditions for such theory. Here we aim at comparing the performance of the different preconditioners introduced in the previous sections, for such discretizations.\\

More precisely, we consider the following model problem on the unit square with domain $\Omega = [0,1]^2 $:
\begin{equation*}
 -\Delta u^{\ast} = f \quad \mbox{ in   } \Omega, \qquad u=0 \quad \mbox{on   } \partial\Omega\;,
 \end{equation*}
where the right hand side  $ f $ is chosen so the exact solution is $ u^{\ast} = \sin(\pi x) \sin(\pi y)$, and  we focus  on the  Incomplete Interior Penalty Discontinuous Galerkin (IIPG)  \cite{IIPG04} discretization of the above model problem, with linear discontinuous finite element space  (denoted by $V^{DG}$) on a shape regular triangulation of $ \Omega$, denoted by $\mathcal{T}_h$. The resulting method reads: {\it Find  $u \in V^{DG}$} such that
\begin{equation*}
a_h(u,v)=\int_{\Omega} fv \, dx \quad \forall \, \, v\in V^{DG}\;.
\end{equation*}
The bilinear form of the IIPG method is given by
\begin{equation}\label{iipbilinear}
  a_h(u,v) = \sum_{K\in \mathcal{T}_h} \int_K \nabla u \cdot \nabla v \, dx - \sum_{e\in \mathcal{E}_h} \int_e \avg{\nabla u } \cdot \jump{v} \, ds  + \sum_{e\in \mathcal{E}_h} \frac{\eta}{|e|} \int_e \jump{u} \cdot \jump{v} \, ds,  
\end{equation}
for all $ u,v \in V^{DG} $.  Here, $K\in \mathcal{T}_h$ refer to an element of the triangulation, $e\subset \partial K$ denotes an edge of the element and we have denoted by $\mathcal{E}_h$ the set of all such edges or skeleton of the partition $\mathcal{T}_h$. We have used the standard definition of the average $\{\{.\}\}$
and jump $\left[\hspace{-0.025in}\left[.\right] 
\hspace{-0.025in}\right]$ operators from \cite{abcm}, and the penalty parameter $ \eta $ is set to $5$ in all the experiments.  We denote by $ A_h $ the matrix representation of the operator associated to the bilinear form \eqref{iipbilinear}, with standard Lagrange basis functions.  As noted in the fist section, with a small abuse of notation, we use the same notation $ A, B^{-1}, Z^{-1} $ and so on for operators and for their representations as matrices.  
Here, ${\bf u}$ and ${\bf f}$ denote the vector representations (in the same basis) of the solution (that we aim to compute) and the right hand side. In the end, the solution process amounts to solve the nonsymmetric system:
\begin{equation}\label{sys:002}
  A_h {\bf u }= {\bf  f}.
\end{equation}

The preconditioners we  use are based on the standard two--level overlapping domain decomposition additive Schwarz preconditioner which we  denote by $ B^{-1} $.  To define it, we consider an overlapping partition of $ \Omega $ into rectangular subdomains $ \Omega_k $ which overlap each other by an amount equal to the fine discretization size $ h $.  Then
\begin{equation}
  B^{-1} = I_H A_H^{-1} I_H^T + \sum_{k=1}^{N_s} I_k A_k^{-1} I_k^T.
  \label{standardasm}
\end{equation}
Here the $ A_k $ operators are restrictions of the original operator $ A_h $ to the finite element space $ V_k $ that is only supported on $ \Omega_k $, that is they correspond to the bilinear forms,
\[
  a_k(u,v) = a_h(u,v), \quad \forall u,v \in V_k,
\]
as in \cite{kara0, barkerldg}.  Since $ V_k \subset V_h $, the operators $ I_k $ are standard injection.  The operator $ A_H $ corresponds to the bilinear form \eqref{iipbilinear} on a coarser discretization of the original domain $ \Omega $, where we label the coarse discretization size $ H $.  We assume that the fine mesh is a refinement of the coarse mesh used to represent $ A_H $ so that $ I_H $ is the natural injection on nested grids.  The penalty parameter on the coarse solver is taken to be $ 5 H/h $ in order to account for the difference of scales in the edge lengths in the penalty terms (see \cite{paola-blanca1, kara0}, for further details).  We implement these preconditioners on a parallel machine where each subdomain of the Schwarz preconditioner corresponds to one processor, where this domain decomposition is in principle independent of the coarse space $V_H$.  The subdomains are square or rectangular, so that on four processors the decomposition is 2 by 2 and on eight processors it is 4 by 2, and so on for larger processor counts.  The tests were run on a commodity Linux cluster using PETSc as the software framework and direct solves for the subdomain problems.

Another preconditioner we consider is
\begin{equation}
  Z^{-1} = B^{-T} A_0 B^{-1}, 
  \label{zdefinition}
\end{equation}
as outlined in the analysis above.  Here $ A_0 $ is a symmetric operator corresponding to the bilinear form
\begin{equation}
  a_0(u,v) = \sum_K \int_K \nabla u \cdot \nabla v \, dx + \sum_e \frac{\eta_0}{|e|} \int_e \jump{u} \cdot \jump{v} \, ds,
  \label{a0bilinearform}
\end{equation}
where the penalty parameter $ \eta_0 = 5 $ is the same as it is in the full bilinear form \eqref{iipbilinear}. \\

In what follows we consider four different preconditioning techniques for the nonsymmetric system \eqref{sys:002},  namely
\begin{enumerate}
  \item the standard additive Schwarz preconditioner $ B^{-1} $ \eqref{standardasm} used in a right--preconditioned GMRES algorithm, for comparison with the other options,
  \item the preconditioner $ Z^{-1} = B^{-T} A_0 B^{-1} $ from \eqref{prec:Z} again used in a right--preconditioned GMRES,
  \item the variable--step preconditioner GMRES variant \eqref{b22} that uses combinations of  $ B^{-1} $ and $ Z^{-1} $,
  \item the variable--step preconditioner GMRES variant \eqref{bbt} that uses combinations of  $ B^{-1} $ and $B^{-T}$.
\end{enumerate}

We note that in the third case, if we have applied $ B^{-1} $ to a vector $ {\bf u} $ we can construct $ Z^{-1} {\bf u }$ by applying $ B^{-T} A_0 $ to save ourselves one preconditioner application.  Our comparisons are done with right--preconditioned GMRES because the procedure in Section \ref{sec:4} minimizes the true residual, as in right--preconditioned GMRES, rather than the preconditioned residual.\\

\begin{table}
  \caption{Iterations to convergence for $ h = 2^{-7}, H = 2^{-5} $ with a nearly exact coarse solver and no restart.}
  \label{iterations}
  \begin{center}\begin{tabular}{l|llll}
    $N_s$& $B^{-1}$ and $Z^{-1}$& $B^{-1}$ and $B^{-T}$& $Z^{-1}$& $B^{-1}$ \\ 
    \hline
    4& 15& 16& 38& 16 \\ 
    8& 15& 18& 42& 18 \\ 
    16& 16& 18& 42& 18 \\ 
    32& 16& 18& 44& 18 \\ 
    64& 15& 17& 43& 17 \\ 
    128& 16& 18& 46& 18 \\ 
  \end{tabular}\end{center}
\end{table}

\begin{figure}
  \begin{center}
    \includegraphics[width=0.5\textwidth]{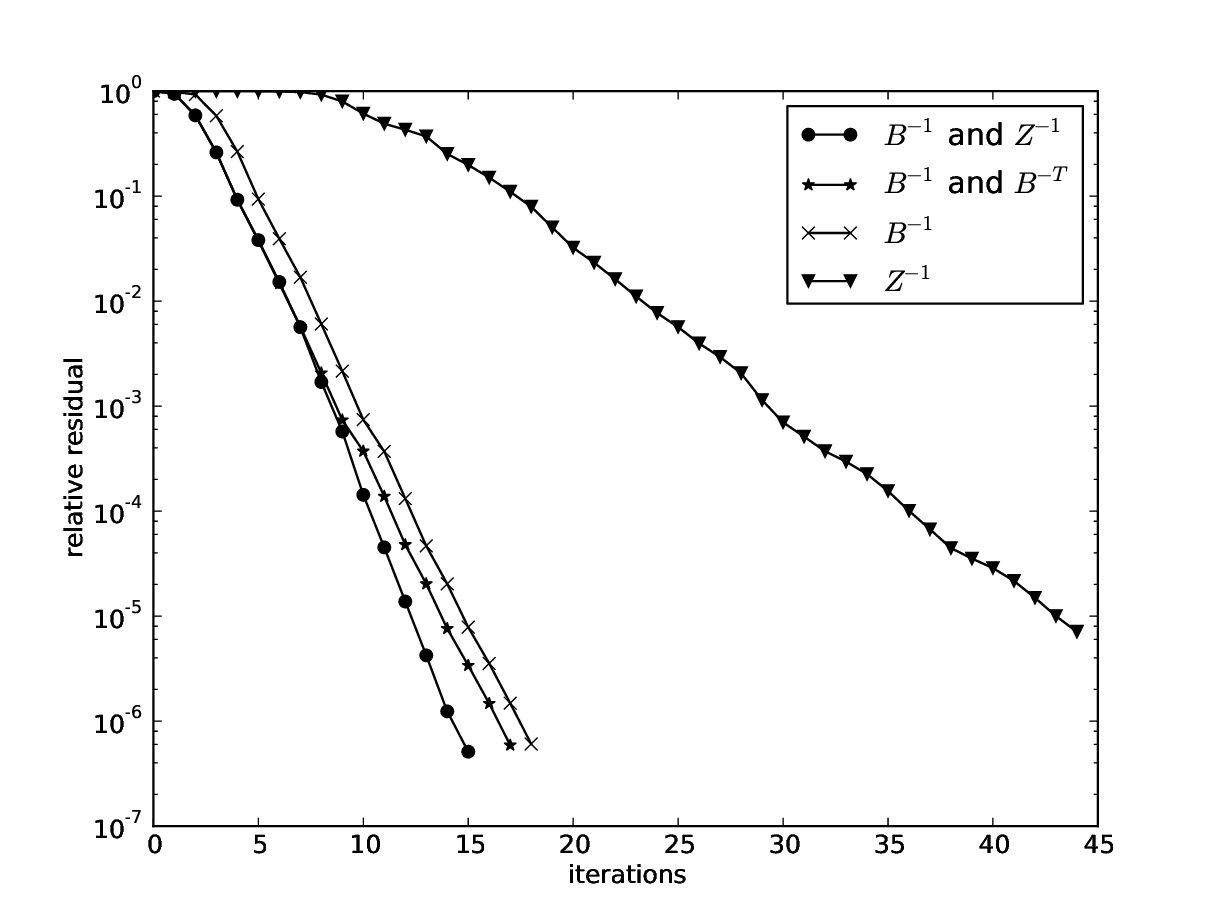}
  \end{center}
  \caption{Convergence of the relative residual ($ \|r_k\| / \| r_0 \| $) on a log scale versus iteration number $ k $ for the different preconditioning strategies, corresponding to the tests in Table \ref{iterations} with $ N_s = 32 $.}
  \label{residualexact}
\end{figure}

\begin{table}
  \caption{Convergence factor for $ h = 2^{-7}, H = 2^{-5} $ with a nearly exact coarse solver and no restart.}
  \label{rate}
  \begin{center}\begin{tabular}{l|llll}
    $N_s$& $B^{-1}$ and $Z^{-1}$& $B^{-1}$ and $B^{-T}$& $Z^{-1}$& $B^{-1}$ \\ 
    \hline
    4& 0.38& 0.41& 0.73& 0.41 \\ 
    8& 0.40& 0.45& 0.76& 0.45 \\ 
    16& 0.40& 0.44& 0.75& 0.44 \\ 
    32& 0.40& 0.45& 0.76& 0.45 \\ 
    64& 0.39& 0.44& 0.76& 0.44 \\ 
    128& 0.39& 0.45& 0.77& 0.45 \\ 
  \end{tabular}\end{center}
\end{table}

The number of GMRES iterations necessary to reduce the relative residual by $ 10^{-6} $ for our four different preconditioning approaches using various numbers of subdomains $ N_s $ for a fixed problem is given in Table \ref{iterations}.  Here we solve the coarse problem involving $ A_H^{-1} $ to a tolerance of $ 10^{-10} $ so that this solve is nearly exact in order to satisfy the theory more closely.  The preconditioning techniques are seen to be scalable in the sense that the number of iterations does not increase as $ N_s $ increases for all four methods (see \cite[Def. 1.3]{toselli-libro}).  The corresponding convergence curves for the case $ N_s = 32 $ are shown in Figure \ref{residualexact}.  We show the convergence factors for this test in Table \ref{rate}, where the convergence factor is defined as 
\[
  \rho = \left( \frac{ \| r_k \| }{ \| r_0 \| } \right)^{1/k},
\]
for the true residuals $ r_k $ where $ k $ is the number of iterations.  

To get a rough idea of computational cost, we show the time to solution in seconds for the four approaches in Table \ref{time}.  We conclude that the $ Z^{-1} $ preconditioner is not competitive because it is the most expensive in terms of time per iteration and it also requires the most iterations.  For this reason we do not consider it further in these numerical results.  The two preconditioning techniques that use the variable--step preconditioner GMRES variant are seen to be effective in convergence rate but to be somewhat more expensive than the classical $ B^{-1} $ preconditioner, as we might expect.  The timings presented in this table are intended to compare the computational costs of the preconditioning strategies and not to demonstrate high performance on parallel machines, which is not evident unless we consider a much larger problem. \\

\begin{table}
  \caption{Time to solution for $ h = 2^{-7}, H = 2^{-5} $ with a nearly exact coarse solver and no restart.}
  \label{time}
  \begin{center}\begin{tabular}{l|llll}
    $N_s$& $B^{-1}$ and $Z^{-1}$& $B^{-1}$ and $B^{-T}$& $Z^{-1}$& $B^{-1}$ \\ 
    \hline
    4& 1.74& 1.79& 3.12& 0.84 \\ 
    8& 1.39& 1.57& 3.22& 0.74 \\ 
    16& 0.70& 0.75& 1.58& 0.34 \\ 
    32& 1.66& 1.83& 4.63& 0.94 \\ 
    64& 1.44& 1.66& 4.65& 0.88 \\ 
    128& 2.70& 2.98& 8.59& 1.76 \\ 
  \end{tabular}\end{center}
\end{table}

\begin{table}
  \caption{Iterations to convergence for $ h = 2^{-7}, H = 2^{-5} $ with an inexact coarse solver and restarting every 10 iterations.}
  \label{practicaliterations}
  \begin{center}\begin{tabular}{l|lll}
    $N_s$& $B^{-1}$ and $Z^{-1}$& $B^{-1}$ and $B^{-T}$& $B^{-1}$ \\ 
    \hline
    4& 15& 18& 21 \\ 
    8& 16& 18& 20 \\ 
    16& 16& 18& 20 \\ 
    32& 16& 18& 19 \\ 
    64& 16& 18& 19 \\ 
    128& 16& 18& 19 \\ 
  \end{tabular}\end{center}
\end{table}

\begin{figure}
  \begin{center}
    \includegraphics[width=0.5\textwidth]{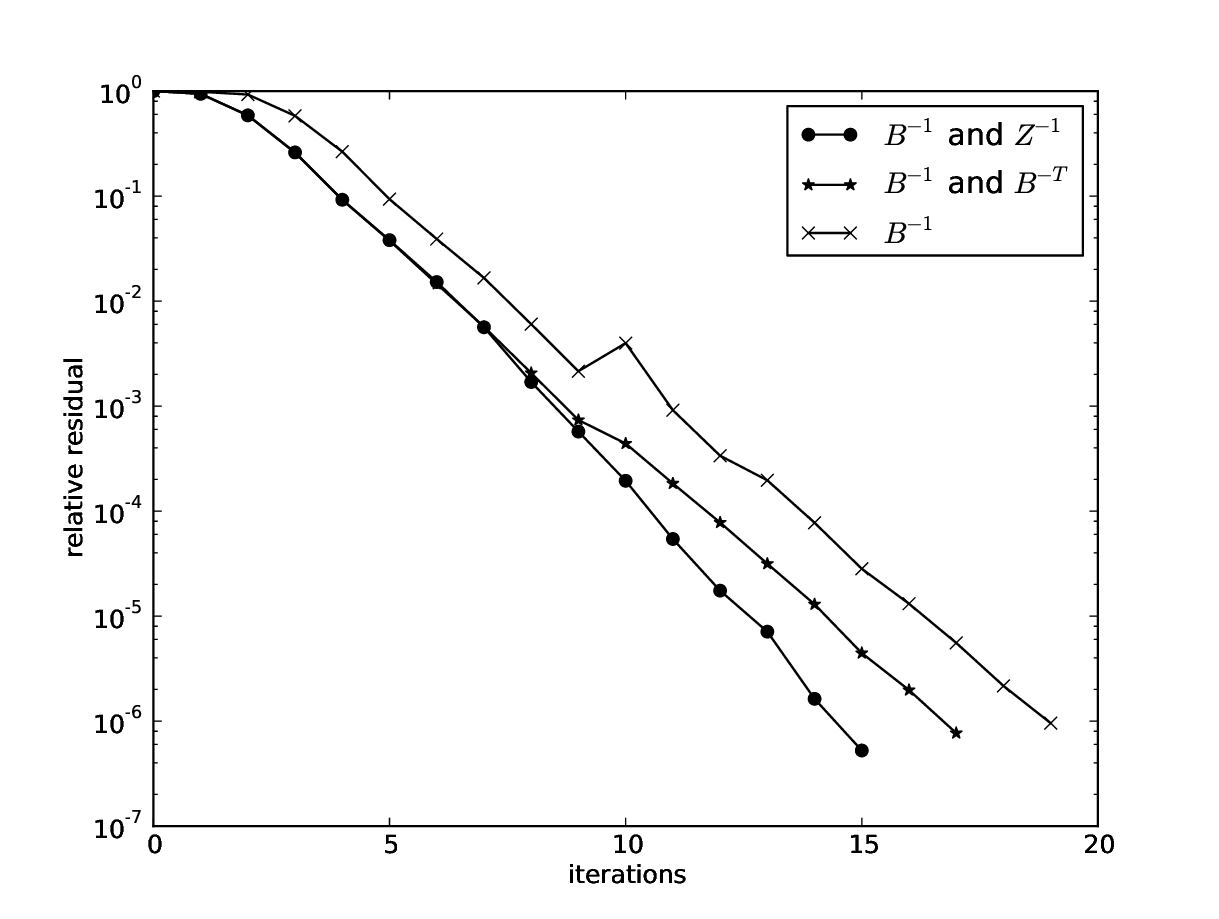}
  \end{center}
  \caption{Convergence of the relative residual ($ \|r_k\| / \| r_0 \| $) on a log scale versus iteration number $ k $ for the different preconditioning strategies, corresponding to the tests in Table \ref{practicaliterations} with $ N_s = 32 $.}
  \label{residualinexact}
\end{figure}

In practice for parallel computing applications the coarse solve in \eqref{standardasm} would not be done exactly.  Another modification that is often made in practice is to restart GMRES after several iterations.  In Tables \ref{practicaliterations}, \ref{practicalrate}, and \ref{practicaltime} we repeat the previous experiment where the relative residual tolerance for the coarse solves is set to $ 10^{-4} $ and GMRES is restarted every 10 iterations.  (These tables should be compared to Tables \ref{iterations}, \ref{rate}, and \ref{time} respectively.)  We see that the convergence behavior is quite similar and the computational cost is lower, suggesting that these common modifications are also useful and effective for our preconditioning strategies.  Corresponding convergence curves are shown in Figure \ref{residualinexact}.\\

\begin{table}
  \caption{Convergence factor for $ h = 2^{-7}, H = 2^{-5} $ with an inexact coarse solver and restarting every 10 iterations.}
  \label{practicalrate}
  \begin{center}\begin{tabular}{l|lll}
    $N_s$& $B^{-1}$ and $Z^{-1}$& $B^{-1}$ and $B^{-T}$& $B^{-1}$ \\ 
    \hline
    4& 0.39& 0.46& 0.51 \\ 
    8& 0.41& 0.46& 0.49 \\ 
    16& 0.41& 0.46& 0.48 \\ 
    32& 0.41& 0.46& 0.48 \\ 
    64& 0.39& 0.45& 0.46 \\ 
    128& 0.40& 0.46& 0.48 \\ 
  \end{tabular}\end{center}
\end{table}

\begin{table}
  \caption{Time to solution for $ h = 2^{-7}, H = 2^{-5} $ with an inexact coarse solver and restarting every 10 iterations.}
  \label{practicaltime}
  \begin{center}\begin{tabular}{l|lll}
    $N_s$& $B^{-1}$ and $Z^{-1}$& $B^{-1}$ and $B^{-T}$& $B^{-1}$ \\ 
    \hline
    4& 1.59& 1.73& 0.99 \\ 
    8& 1.29& 1.35& 0.72 \\ 
    16& 0.59& 0.62& 0.35 \\ 
    32& 0.73& 0.79& 0.46 \\ 
    64& 0.67& 0.71& 0.44 \\ 
    128& 0.89& 0.95& 0.60 \\ 
  \end{tabular}\end{center}
\end{table}

To see how these methods scale to larger problems, we consider a much finer mesh in Tables \ref{fineriterations}, \ref{finerrate}, and \ref{finertime}, while keeping the mesh size for the coarse solve quite coarse.  The scalability of the preconditioners in terms of iterations is still present, the preconditioner still performs well, and in these cases we can see fairly good parallel scalability in the sense that for a fixed problem size, doubling the number of processors in the parallel solve roughly cuts the execution time in half for all our preconditioning strategies.  Again convergence curves for this case are shown in Figure \ref{residualsmallh}.\\

\begin{table}
  \caption{Iterations to convergence for $ h = 2^{-10}, H = 2^{-6} $ with an inexact coarse solver and restarting every 10 iterations.}
  \label{fineriterations}
  \begin{center}\begin{tabular}{l|lll}
    $N_s$& $B^{-1}$ and $Z^{-1}$ & $B^{-1}$ and $B^{-T}$ & $B^{-1}$ \\ 
    \hline
    32& 30& 32& 31 \\ 
    64& 30& 31& 32 \\ 
    128& 30& 31& 31 \\ 
    256& 30& 31& 30 \\ 
  \end{tabular}\end{center}
\end{table}

\begin{figure}
  \begin{center}
    \includegraphics[width=0.5\textwidth]{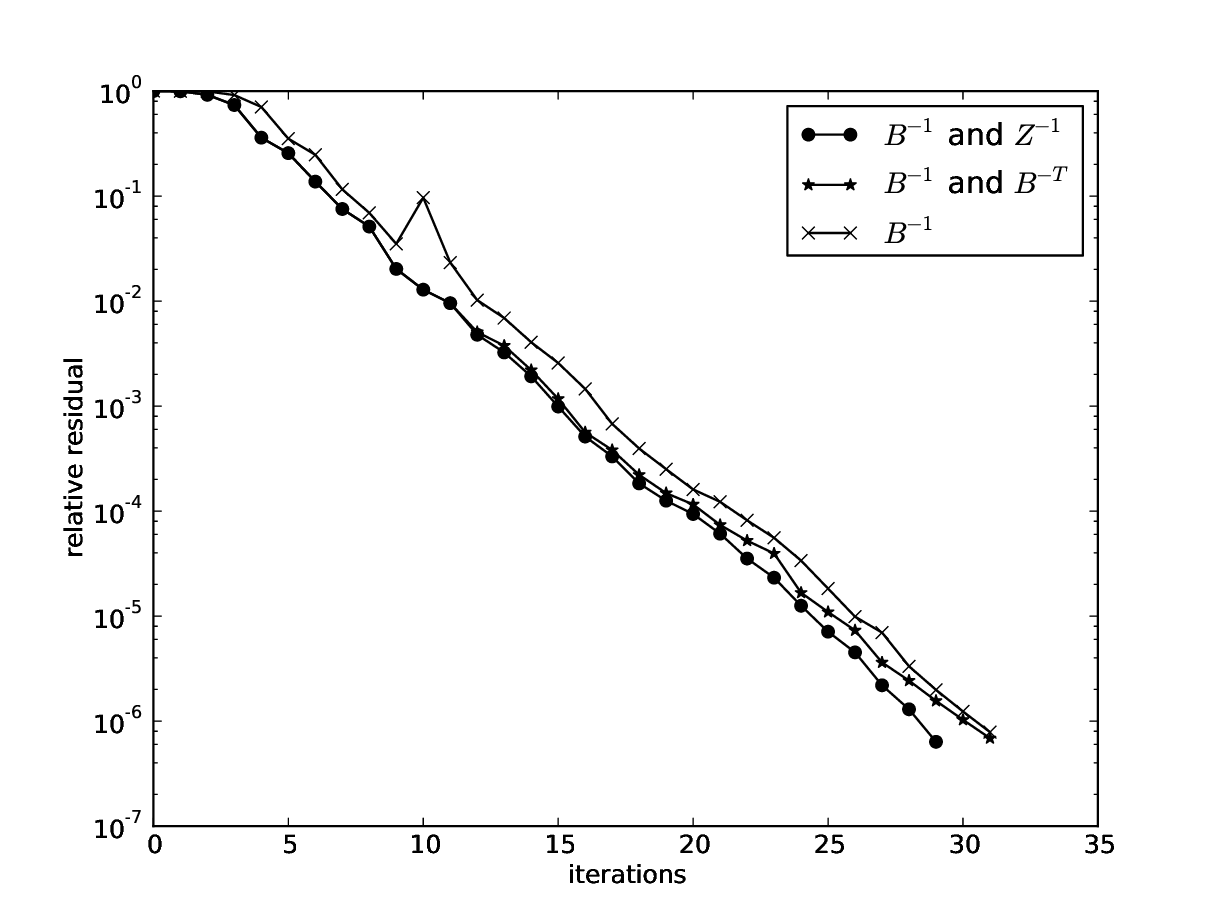}
  \end{center}
  \caption{Convergence of the relative residual ($ \|r_k\| / \| r_0 \| $) on a log scale versus iteration number $ k $ for the different preconditioning strategies, corresponding to the tests in Table \ref{fineriterations} with $ N_s = 128 $.}
  \label{residualsmallh}
\end{figure}

\begin{table}
  \caption{Convergence factor for $ h = 2^{-10}, H = 2^{-6} $ with an inexact coarse solver and restarting every 10 iterations.}
  \label{finerrate}
  \begin{center}\begin{tabular}{l|lll}
    $N_s$& $B^{-1}$ and $Z^{-1}$& $B^{-1}$ and $B^{-T}$& $B^{-1}$ \\ 
    \hline
    32& 0.62& 0.64& 0.64 \\ 
    64& 0.63& 0.64& 0.64 \\ 
    128& 0.62& 0.63& 0.63 \\ 
    256& 0.63& 0.63& 0.63 \\ 
  \end{tabular}\end{center}
\end{table}

\begin{table}
  \caption{Time to solution for $ h = 2^{-10}, H = 2^{-6} $ with an inexact coarse solver and restarting every 10 iterations.}
  \label{finertime}
  \begin{center}\begin{tabular}{l|lll}
    $N_s$& $B^{-1}$ and $Z^{-1}$ & $B^{-1}$ and $B^{-T}$ & $B^{-1}$ \\ 
    \hline
    32& 94.32& 95.55& 40.96 \\ 
    64& 37.89& 37.59& 16.66 \\ 
    128& 15.55& 15.60& 7.82 \\ 
    256& 7.76& 8.81& 5.50 \\ 
  \end{tabular}\end{center}
\end{table}

In all the results we have presented so far, the variable--step preconditioner GMRES variant has performed slightly better than the classical $ B^{-1} $ preconditioner in terms of number of iterations to convergence, but the increased computational cost per iteration of the combined preconditioned GMRES method has ended up making the classical preconditioner perform better in execution time.  This suggests that the new method may be competitive in settings where each iteration is very expensive, so that the savings in iteration count can make up for the increased cost per iteration.  To examine this setting we consider a problem in Tables \ref{competitiveiterations} and \ref{competitivetime} where the coarse grid solve is done on a relatively fine grid and is therefore quite expensive.  The results in this somewhat artificial setting do show that the new methods are competitive with the classical preconditioning techniques in terms of computational cost.\\

\begin{table}
  \caption{Iterations to convergence for $ h = 2^{-10}, H = 2^{-9} $ with an inexact coarse solver and restarting every 10 iterations.}
  \label{competitiveiterations}
  \begin{center}\begin{tabular}{l|lll}
    $N_s$& $B^{-1}$ and $Z^{-1}$ & $B^{-1}$ and $B^{-T}$ & $B^{-1}$ \\ 
    \hline
    32& 19& 20& 23 \\ 
    64& 17& 18& 22 \\ 
    128& 15& 17& 21 \\ 
    256& 15& 16& 20 \\ 
  \end{tabular}\end{center}
\end{table}

\begin{table}
  \caption{Time to solution for $ h = 2^{-10}, H = 2^{-9} $ with an inexact coarse solver and restarting every 10 iterations.}
  \label{competitivetime}
  \begin{center}\begin{tabular}{l|lll}
    $N_s$& $B^{-1}$ and $Z^{-1}$ & $B^{-1}$ and $B^{-T}$ & $B^{-1}$ \\ 
    \hline
    32& 277.50& 280.69& 352.85 \\ 
    64& 170.94& 187.29& 133.73 \\ 
    128& 79.01& 93.61& 99.50 \\ 
    256& 40.92& 45.41& 39.59 \\ 
  \end{tabular}\end{center}
\end{table}

As a final test, we consider the variable--step preconditioner  GMRES method using \eqref{b22} for solving a nonsymmetric convection--diffusion problem. The model problem is  set on $ \Omega=[0,1]^2 $ and given by
\[
  -\epsilon \Delta u^\ast + \mathbf b \cdot \nabla u^\ast = f \quad \mbox{ in   } \Omega, \qquad u=0 \quad \mbox{on   } \partial\Omega\;.
\] 
In our tests we take $ \mathbf b = (2,1)^T $ and $ \epsilon = 0.01 $. The data $f$ is chosen so that the exact solution is 
\[
  u^\ast = \left[ x + \frac{e^{b_1 x/\epsilon} - 1}{1 - e^{b_1/\epsilon}} \right] \cdot \left[ y + \frac{e^{b_2 y/\epsilon} - 1}{1 - e^{b_2/\epsilon}} \right]
\]
which has a boundary layer at the top and right sides of the domain. For this problem we use the symmetric interior penalty method (SIPG) \cite{arnold82,abcm} for discretizing the elliptic operator and standard upwind for the convection term. The nonsymmetry here comes from the discretization of the convective term.

The results we present are for the combined GMRES with the two preconditioners $B^{-1}$ and $Z^{-1}$, but for $Z^{-1}$ we need to define a symmetric positive definite auxiliary matrix $A_0$.  To try to include some of the convective terms we define $A_0$ to correspond to the bilinear form
\[
  a_0^c(u,v) = \epsilon a_0(u,v) + \sum_e \int_e | \mathbf b \cdot \mathbf n | \jump{u} \cdot \jump{v} \, ds,
\]
where $ a_0(\cdot,\cdot) $ is the symmetrized form previously defined in \eqref{a0bilinearform}.

Results for this example are shown in Tables \ref{convectioniterations} and \ref{convectionrate}, where we see that the convergence of the variable--step preconditioner performs better than the usual Schwarz preconditioning. We do not present timing for this example because we used a different computer than for the previous results, but qualitatively the timing comparison is similar to that in Table \ref{time}. 

\begin{table}
  \caption{For the convection--diffusion problem, number of iterations to convergence with $ h = 2^{-7}, H = 2^{-5}, \epsilon=0.01 $, using a nearly exact coarse solver and no restart.}
  \label{convectioniterations}
  \begin{center}\begin{tabular}{l|ll}
    $N_s$& $B^{-1}$& $B^{-1}$ and $Z^{-1}$ \\ 
    \hline
    4& 18& 16 \\ 
    8& 19& 17 \\ 
    16& 19& 17 \\ 
    32& 20& 18 \\ 
    64& 20& 18 \\ 
    128& 20& 19 \\ 
  \end{tabular}\end{center}
\end{table}

\begin{table}
  \caption{For the convection--diffusion problem, convergence factor with $ h = 2^{-7}, H = 2^{-5}, \epsilon=0.01 $, using a nearly exact coarse solver and no restart.}
  \label{convectionrate}
  \begin{center}\begin{tabular}{l|ll}
    $N_s$& $B^{-1}$& $B^{-1}$ and $Z^{-1}$ \\ 
    \hline
    4& 0.46& 0.41 \\ 
    8& 0.46& 0.43 \\ 
    16& 0.47& 0.44 \\ 
    32& 0.49& 0.46 \\ 
    64& 0.49& 0.46 \\ 
    128& 0.49& 0.48 \\ 
  \end{tabular}\end{center}
\end{table}

\section*{Acknowledgments}

The authors thank the referees for helpful comments and suggestions that have led to an improved version of the paper.

 \bibliographystyle{plain}

\end{document}